\documentclass[a4paper,12pt]{article}
\usepackage[warn]{mathtext}
\usepackage{cmap}
\usepackage[T2A]{fontenc}
\usepackage[cp1251]{inputenc}
\usepackage[russian]{babel}
\usepackage{amsmath}
\usepackage{wasysym}
\usepackage{amssymb,amsthm}
\usepackage[dvips]{graphicx}

\theoremstyle{plain}
\newtheorem{theorem}{Теорема}[section]
\newtheorem{corollary}[theorem]{Следствие}
\newtheorem{lemma}[theorem]{Лемма}
\newtheorem{proposition}[theorem]{Предложение}

\newtheorem{remark}[theorem]{Замечание}

\DeclareMathOperator{\Evol}{Evol}
\DeclareMathOperator{\cut}{cut}

\theoremstyle{definition}
\newtheorem{definition}[theorem]{Определение}

\begin{document}
\title{Равномерная рекуррентность HD0L систем.}

\author{И.~Митрофанов}

\maketitle

\begin{center}
\bigskip
Московский Государственный Университет им. М.~В.~Ломоносова.
\end{center}
\begin{abstract}
We prove that the problem of deciding whether a given morphic sequence is uniformly recurrent is decidable.
The proof uses decidability of HD0L periodicity problem, which was recently proved in papers of F.Durand and I.Mitrofanov.
\end{abstract}

\section{Введение.}
{\it Алфавит} -- это конечное множество элементов, называемых {\it буквами}.

Конечная последовательность букв алфавита $A$ называется конечным словом (или просто словом) над $A$. Множество конечных слов, включая пустое, обозначается $A^*$. Пустое слово обозначают $\varepsilon $.
{\it Бесконечное слово}, или {\it сверхслово} -- это отображение из $\mathbb N$ в алфавит $A$. Множество сверхслов обозначается $A^{\omega}$.

Для конечного слова определена {\it длина} -- количество букв в нём. Длина слова $u$ будет обозначаться $|u|$.
Если слово $u_1$ конечно, то определена {\it конкатенация} слов $u_1$ и $u_2$ -- слово $u_1u_2$, получающееся приписыванием второго к первому справа.

Слово $v$ является {\it подсловом} слова $u$, если $u=v_1vv_2$ для некоторых слов $v_1$, $v_2$. В случае, когда $v_1$ или $v_2$ -- пустое слово, $v$ называется {\it началом} или соответственно {\it концом} слова $u$. На словах существует естественная структура частично упорядоченного множества: $u_1\sqsubseteq u_2$, если $u_1$ является подсловом $u_2$. Будем обозначать $u_1\sqsubseteq_k u_2$, если слово $u_1$ входит в $u_2$ не менее $k$ раз.

Сверхслово $W$ называется {\it рекуррентным}, если любое его подслово встречается в $W$ бесконечно много раз, иначе говоря, $v\sqsubseteq W\Rightarrow v\sqsubseteq _{\infty}W$.

Сверхслово $W$ называется {\it равномерно рекуррентным}, если для любого его подслова $u$ конечной длины существует такое число $N(u)$, 
что для любого подслова сверхслова $W$ длины $N(u)$ в этом подслове есть хотя бы одно вхождение $u$ (следовательно, $W$ является рекуррентным).

Сверхслово $W$ называется {\it периодичным}, если $W=uuuuu\dots$ для некоторого непустого $u$. Само слово $u$, равно как и его длина, называются
{\it периодом} сверхслова $W$. 
Сверхслово называется {\it заключительно периодичным}, если оно представляется в виде конкатенации $W=uW'$, где $u$ -- конечное слово, а $W'$ --
периодичное сверхслово.
\begin{proposition} Если сверхслово заключительно периодично и рекуррентно, то оно периодично.
\end{proposition}

Множество слов $A^{*}$ над алфавитом $A$ можно считать свободным моноидом с операцией конкатенацией и единицей -- пустым словом. Отображение $\varphi\colon A^{*}\to B^{*}$ называется {\it морфизмом}, если оно сохраняет операцию моноида. Очевидно, морфизм достаточно задать на буквах алфавита $A$. Морфизм называется {нестирающим}, если образом никакой буквы не является пустое слово. Если $|\varphi(a_i)|=1$ для любой буквы $a_i\in A$,
то $\varphi$ -- {\it кодирование}. Если алфавиты $A$ и $B$ совпадают, то $\varphi$ называется {\it подстановкой}.

Морфизм можно продолжить на бесконечное слово по правилу: если $u$ -- начало $W$, то $\varphi(u)$ -- начало $\varphi(W)$.

Если $\varphi$ -- такая подстановка, что $\varphi(a_1)=a_1v$ для некоторого слова $v$ и для всех $k\in \mathbb N$ $\varphi^k(v)\not=\varepsilon$, то говорится, что $\varphi$ {\it продолжается над} $a_1$ и бесконечное слово 
$$
\varphi^{\infty}(a_1):=a_1v\varphi(v)\varphi^2(v)\varphi^3(v)\varphi^4(v)\dots
$$
называется {\it чисто морфическим}, или {\it бесконечной неподвижной точкой} морфизма $\varphi$.
Если задан морфизм $\psi\colon A^{*}\to C^{*}$, то сверхслово $\psi(\varphi^{\infty}(a_1))$ называется {\it морфическим}.

Сформулируем {\it проблему равномерной рекуррентности для морфических слов}, поставленную в \cite{Pr} и \cite{PS}:

{\bf Дано:} два конечных алфавита $A$ и $B$, буква $a_1\in A$, подстановка $\varphi:A^*\to A^*$, 
продолжающаяся над $a_1$ и морфизм $\psi:A^*\to B^*$.

{\bf Определить:} является ли слово $\psi(\varphi^{\infty}(a_1))$ равномерно рекуррентным?

В настоящей работе строится алгоритм для произвольного морфического слова:
\begin{theorem} \label{main}
Проблема равномерной рекуррентности для морфических слов разрешима.
\end{theorem}

Этот результат был анонсирован в \cite{MKR}. До этого проблема была решена в случае автоматных последовательностей \cite{Ch}, в работе \cite{Pr}
строится полиномиальный алгоритм для классов чисто морфических и автоматных слов.

\section{Приведение морфизмов к удобному виду.}
Сформулируем {\it проблему заключительной периодичности для морфических слов.}

{\bf Дано:} два конечных алфавита $A$ и $B$, буква $a_1\in A$, подстановка $\varphi:A^*\to A^*$, 
продолжающаяся над $a_1$ и морфизм $\psi:A^*\to B^*$.

{\bf Определить:} является ли слово $\psi(\varphi^{\infty}(a_1))$ заключительно периодическим. Если является, явно указать его период.

\begin{theorem}\label{MP}
Проблема заключительной периодичности для морфических слов разрешима.
\end{theorem}

См. работы \cite{MP,Dur}.

\begin{theorem}[доказательство см., например, в \cite{MP}] \label{monad}
 Пусть $\varphi$ -- подстановка, действующая на алфавите $A$, $\psi$ -- морфизм из $A^*$ в $B^*$, $u$ -- конечное слово из $B^*$.
Тогда существует алгоритм, проверяющий, встречается ли $u$ в слове $W=\psi(\varphi^{\infty}(a_1))$ и, если встречается, конечное ли число раз.
\end{theorem}

\begin{proposition}\label{localP}
Если равномерно рекуррентное сверхслово $W$ для некоторого непустого $U$ и любого натурального $k$ содержит подслово $U^k$, то оно является число периодическим
с периодом $u$.
\end{proposition}
\begin{proof} Это следует из того, что если $|V|\sqsubseteq W$ и $|V|=|U|$, то $V$ является циклическим сдвигом слова $U$. 
\end{proof}

Мы хотим проверить на равномерную рекуррентность сверхслово $W=\psi(\varphi^{\infty}(a_1))$. Морфизм $\varphi$ можно считать нестирающим, а $h$ -- кодированием (то есть для любой буквы $a\in A$ $|\varphi(a)|>0$, $|\psi(a)|=1$).  Это следует из результата
\begin{theorem}[см. \cite{AS}, глава $7$] Если $f: A^*\to B^*$ и $g:A^*\to A^*$ -- произвольные морфизмы и $f(g^{\infty}(a_1))$ --
бесконечное слово, то можно найти такие алфавит $A'$, букву $a'_1\in A'$, нестирающую подстановку $\varphi$, действующую на алфавите $A$
и кодирование $\tau:A'\to B$, что $f(g^\infty(a_1))=\tau(\varphi^{\infty}(a'_1))$.
\end{theorem}

Слово $w\in A^*$ будем называть $\varphi-${\it ограниченным}, если последовательность
$$
w,\varphi(w),\varphi^2(w),\varphi^3(w),\dots
$$
 периодична начиная с некоторого момента.
В противном случае, $|\varphi^n(w)|\rightarrow \infty$ при $n\rightarrow \infty$ и слово $w$ называется {\it $\varphi-$растущим}.
Очевидно, слово является $\varphi-$ограниченным тогда и только тогда, когда оно состоит из $\varphi-$ограниченных букв.

\begin{theorem} \label{finwords}
Существует алгоритм, который определяет, конечно ли в $\varphi^{\infty}(a)$ число различных $\varphi-$ограниченных подслов.
Если это число бесконечно, то можно указать такое непустое $U$, что $U^k$ является подсловом $\varphi^{\infty}(a)$ для любого $k\in \mathbb N$.
Если оно конечно, то все $\varphi-$ограниченные слова алгоритмически находятся.
\end{theorem}

\begin{proof}
Прежде всего отметим, что все $\varphi-$возрастающие буквы алгоритмически находятся (см., например, \cite{Pr}). Далее $\varphi-$растущие буквы будем писать как $a_1$, $a_2$ и т.д.

Построим ориентированный граф $Q$, на рёбрах которого будут записаны упорядоченные пары слов.
Вершинами этого графа будут служить $\varphi-$растущие буквы из $A$ а также всевозможные упорядоченные пары $\varphi -$растущих букв. Введём фиктивную букву $t$, также к вершинам $Q$ добавим всевозможные пары вида $a_it$, где $a_i$ -- $\varphi-$растущая буква.

Из вершины $a_i$ в $a_j$ идёт ребро, если $a_j\in a_i$. На таких рёбрах пара слов -- $\{\varepsilon ,\varepsilon \}$ ($\varepsilon $-- пустое слово).
Из вершины $a_i$ в $a_ja_k$
ведёт ребро со словами $\{\omega,\varepsilon\}$, если для некоторого
$\varphi-$ограниченного слова $\omega $ слово $a_j\omega a_k$ является подсловом $\varphi(a_i)$ (из $a_i$ в $a_ja_k$ могут вести несколько рёбер.)
Из $a_i$ и $a_it$ ведёт по ребру с парой $\{\omega,\varepsilon\}$ в $a_jt$, если $\omega$ -- $\varphi-$ограниченное и $\varphi(a_i)$ оканчивается на $a_j\omega$.

Из $a_ia_j$ ведёт ребро в $a_ka_l$ с парой $\varphi-$ограниченныx слов $\{\omega_1;\omega_2\}$, если $\varphi(a_1)$ кончается на $a_k\omega_1$, а
$\varphi(a_2)$ начинается на $\omega_2a_l$.
\begin{proposition}
Пусть $k\in\mathbb N$. Рассмотрим какой-нибудь путь длины $k$ по рёбрам графа $Q$, выходящий из $a_1$.
Последовательность пар слов на рёбрах этого пути
$$
\{u_1,v_1\},\{u_2,v_2\},\dots,\{u_k,v_k\}. 
$$
Тогда в $\varphi^k(a_1)$ есть $\varphi-$нерасширяемое слово
$$
u_k\varphi (u_{k-1})\dots\varphi ^{k-1}(u_1)\varphi^{k-1}(v_1)\varphi^{k-2}(v_2)\dots v_k.
$$
При этом, если путь оканчивается на $a_ia_j$, то есть вхождение этого слова, обрамлённое буквами $a_i$ и $a_j$.

Наоборот, любое $\varphi-$ограниченное подслово $\varphi^k(a_1)$ можнно получить, получив слово по указанному правилу слово и взяв его подслово.
\end{proposition}

Это несложно показывается индукцией по $k$.

Первый случай: в любом ориентированном цикле графа $Q$, до которого можно добраться из $a_1$, на рёбрах цикла написаны пары пустых слов. Тогда в любом пути, выходящем из вершины $a_1$, число рёбер, на которых написаны не пустые слова, не превосходит количества вершин в $Q$. Следовательно, число различных $\varphi-$ограниченных слов конечно, а по графу $Q$ можно их всех найти.

Второй случай: есть цикл, и в этом цикле не все слова пустые. Пусть, например, в цикле есть пары слов, в которых первое слово не пустое. Тогда для 
некоторой буквы $a_i$ и непустого $\varphi-$ограниченного $u$, слово $\varphi(a_i)$ оканчивается на $a_iu$.
Образ $u$ при подстановке $\varphi$ зациклится. Следовательно, для некоторого непустого слова $U$ для любого $k\in\mathbb N$ слово $U^k$ является
подсловом $\varphi^{\infty}(a)$.
\end{proof}

Применим этот алгоритм к $\varphi$ и $a_1$. Если некоторое слово $U$ повторяется в $\varphi(a_1)$ сколь угодно много раз подряд,
то $\psi(U)$ повторяется сколь угодно много раз подряд в $W$. Согласно \ref{localP}, для установления равномерной рекуррентности $W$ достаточно проверить, является ли оно чисто периодическим с периодом $\psi(U)$.
\begin{proposition}
Бесконечное слово является чисто периодическим с заданным периодом $A$ тогда и только когда, когда все его конечные подслова длины $|A|$ являются циклическими сдвигами $A$.
\end{proposition}
Согласло \ref{monad}, все конечные подслова $W$ длины $|\psi(U)|=|U|$ можно найти, стало быть, в этом случае определить, является ли 
$W$ равномерно рекуррентным, мы можем.

В дальнейшем рассматриваем случай, когда в $\varphi^{\infty}(a_1)$ конечное число $\varphi-$ ограниченных подслов.

Пусть $I_{\varphi}$ -- множество всех $\varphi-$растущих букв, $B_{\varphi}$ -- множество $\varphi-$ограниченных подслов сверхслова
$\varphi^{\infty}(a_1)$(включая пустое слово).
Можно считать, что $B_{\varphi}$ конечно и что мы знаем все слова в $B_{\varphi}$.
Рассмотрим (конечный) алфавит $C$, состоящий из символов $[twt']$, где $t$ и $t'$ буквы из $I_{\varphi}$, а 
$w$ -- слово из $B_{\varphi}$ и слово $twt'$ является подсловом $\varphi^{\infty}(a)$.

Определим морфизм $\varphi':C^* \to C^*$ следующим образом:
$$
\varphi'([twt'])=[t_1wt_2][t_2wt_3]\dots[t_kw_kt_{k+1}],
$$
где $\varphi(tw)=w_0t_1w_1t_2\dots t_kw'_k$, слово $\varphi(t')$ начинается с $w''_kt_{k+1}$ и $w_k=w'_kw''_k$ (cлова $w_i$, $w'_k$ и $w''_k$ принадлежат $B_{\varphi} $).

Также определим $f:C^*\to A^*$ по правилу
$$
f([twt'])=tw.
$$

\begin{proposition} \label{mor_hor}
 Все буквы алфавита $C$ являются $\varphi'-$растущими.
\end{proposition}

\begin{proof}
Заметим, что в $\varphi'^n([twt'])$ столько же букв, сколько в слове $\varphi^n(t)$ $\varphi-$растущих букв.
Очевидно, в образе $\varphi(t)$ от произвольной буквы $t\in I_{\varphi}$ содержится хотя бы одна буква из $I_{\varphi}$. Более того, в слове $\varphi ^n(t)$ для некоторого $n$ содержатся хотя бы две буквы из $I_{\varphi}$, иначе $\varphi ^n(t)=w_nt_{i_n}v_n$
(где $w_n$ и $v_n$ принадлежат $B_{\varphi }$) и $|\varphi^n{t}|$ ограниченно.
\end{proof}

Пусть $\varphi^{\infty}(a)$ имеет вид $a_1w_1a_2\dots$, где $a_1,a_2\in I_{\varphi}$, $w_1\in B_{\varphi}$.
Тогда, несложно убедиться, что для любого $n$ слово $\varphi ^n(a)$ является началом слова $f(\varphi'^n([a_1w_1a_2]))$,
следовательно, слова $W$ и $(\psi\circ f)(\varphi'^{\infty}([a_1w_1a_2]))$ совпадают. Заметим, что морфизм $h'=h\circ f$ является нестирающим.

\begin{remark} Конструкция морфизма $\psi$ встречалась в работах \cite{Pr,Pans}.
\end{remark}

{\bf Таким образом, можно считать, что все буквы алфавита $A$ являются $\varphi-$ растущими, а $h$ - произвольный нестирающий морфизм.}

Букву $a_i\in A$ назовём {\it рекуррентной}, если для некоторого $k\in \mathbb N$ выполнено $a_i\sqsubseteq \varphi^k(a_i)$.
Для каждой рекуррентной буквы алфавита существует такое число $k(a_i)$, что если $k|n$, то $a_i\sqsubseteq \varphi^n(a_i)$.	Следовательно, существует
такое $n\in \mathbb N$, что если $a_i$ -- произвольная рекуррентная буква алфавита, то $a_i\sqsubseteq \varphi^n(a_i)$.
Положим $\rho=\varphi^n$.

Рассмотрим ориентированный граф $G_{\rho}$, вершинами которого являются буквы алфавита $A$, и из $a_i$ ведёт стрелка в $a_j$ тогда и только тогда,
когда $a_j$ содержится в $\rho(a_i)$. 

Пусть $D$ -- сильносвязная компонента этого графа, до которой можно дойти по стрелочкам из $a_1$. Рассмотрим ограничение $\rho $ на $D^*$.
Все буквы из $D$ являются рекуррентными. Следовательно, если $d\in D$, то $d\sqsubseteq \rho(d)$. 
Также для любого $k$ выполнено $\rho^k(d)\sqsubseteq \rho^{k+1}(d)$. 
Следовательно, существует такое $m$, что для любых букв $d_1$ и $d_2\in D$ $d_2\sqsubseteq \rho^m(d_1)$.
Поэтому морфизм $\rho$ в ограничении на $D^*$ является примитивным.

Найдётся такая буква $d\in D$, что $\rho^l(d)$ начинается на $d$ для некоторого $l$. Обозначим $\rho_2=\rho^l$.
Так как все буквы из $D$ являются $\rho_2-$растущими, то $\rho_2^{\infty}(d)$ является бесконечным сверхсловом, все конечные подслова которого являются подсловами $\varphi ^{\infty}(a_1)$.

Слово $H=\psi(\rho_2^{\infty}(d))$ является равномерно рекуррентным как примитивное (см., например, в \cite{AS}.)

\begin{proposition}\label{Pr}
Пусть $W$ и $H$ -- сверхслова, $H$ является равномерно рекуррентным, и все конечные подслова $H$ являются подсловами сверхслова $W$.
Тогда $W$ является равномерно рекуррентным тогда и только тогда, когда любое его конечное подслово является подсловом сверхслова $H$.
\end{proposition}
Доказательство можно найти в \cite{Pr}.

Проверим, является ли $H$ периодичным. Если является, то достаточно проверить периодичность слова $W$, что мы делать умеем.
Поэтому далее считаем, что $H$ -- непериодичное слово.
\section{Порядок роста букв.}

Напомним: $\varphi $ -- продолжающаяся над $a_1$ подстановка на алфавите $A$, для которой все буквы являются $\varphi -$возрастающими,
$\psi:A\to B^+$ -- нестирающий морфизм. Можно считать, что все буквы из $A$ принадлежат $\varphi ^{\infty}(a_1)$.

Известно следующее утверждение (см. \cite{SS}):
\begin{proposition} Пусть $\sigma:A^*\to A^*$ -- подстановка. Для любой буквы $a\in A$ выполнено одно из двух условий:
\begin{enumerate}
\item $\exists k\in \mathbb N: \sigma^k(a)=\varepsilon .$
\item $\exists d(a)\in \mathbb N_0, c(a)\in \mathbb R_+, C(a)\in \mathbb R, \theta(a) \in \mathbb R$ такие, что для всех $k$ выполнено
$$c(a)<\frac{\sigma^n(a)}{c(a)n^{d(a)}\theta(a)^n}<C(a)$$
\end{enumerate}
\end{proposition}

Очевидно, для $\varphi$ выполнено второе условие. Для каждой буквы $a_i$ пара $(d(a_i),\theta(a_i))$ называется {\it порядком роста буквы}. 
Также (см. \cite{Dur,SS}) известно, что если в $A$ есть буква порядка роста $(n,\theta)$, где $n>0$, 
то в $A$ есть буква с порядком роста $(n-1,\theta)$. Таким образом, если для некоторой буквы $a_i\in A$ $\theta(a_i)=1$, то в $A$ есть
буква роста $(0,1)$, то есть $\varphi-$ограниченная. Так как у нас все буквы $\varphi-$растущие, то для любой буквы $a_i$ $\theta(a_i)>1$.  

На порядках роста можно ввести операцию сравнения: $(d_1,\theta_1)<(d_2,\theta_2)$, если $\theta_1<\theta_2$ или $\theta_1<\theta_2$ и $d_1<d_2$.
Если у буквы $a_1$ порядок роста меньше, чем у $a_2$, то $\lim_{k\to\infty}\dfrac{|\varphi ^k(a_1)|}{|\varphi ^k(a_2)|}=0$.
Аналогично порядок роста буквы определяется {\it порядок роста конечного слова}, который для слова $u$ будем обозначать $r(u)$.
\begin{proposition}\label{qs}
Очевидно, $r(a_1a_2\dots a_n)=\max_{i\in[1;n]}r(a_i)$. Кроме того, $r(u)=r(\varphi(u))$.
\end{proposition}

\begin{lemma}\label{rpart}
Пусть $r(u)=(d,\theta)$, при этом $\theta>1$. Тогда для некоторого $C(u)$ для всех $k$ выполнено
$$|u|+|\varphi (u)|+|\varphi^2(u)|+\dots +|\varphi^k(u)|<C(u)k^d\theta^k.$$
\end{lemma}
\begin{proof}
В самом деле, для некоторого $C$ для любого $k$ выполнено $|\varphi^k(u)|<Ck^d\theta ^k$. Тогда 
$\sum_{i=0}^k|\varphi^i(u)|<C\sum_{i=0}^ki^d\theta^i<C\sum_{i=0}^kk^d\theta^i<\dfrac{\theta C}{\theta-1}k^d\theta^k$.
\end{proof}

Положим $D,\Theta$ -- самый большой порядок роста среди порядков роста для всех букв из алфавита. 
Буквы с таким порядком роста будем называть {\it быстрорастущими.} Очевидно, $a_1$ является быстрорастущей буквой.
\begin{lemma}
В слове $\varphi^{\infty}(a_1)$ бесконечно много быстрорастущих букв.
\end{lemma}
\begin{proof}
Пусть $\varphi (a_1)=a_1u$. Тогда $\varphi^{\infty}(a_1)=a_1u\varphi (u)\varphi ^2(u)\varphi ^3(u)\dots$

Если в слове $u$ есть хотя бы одна быстрорастущая буква, то хотя бы одна быстрорастущая буква есть в $\varphi ^k(u)$ для любого $k$.

Предположим, что в $u$ быстрорастущих букв нет. Пусть $r(u)=(d,\theta)<(D,\Theta)$. Как мы знаем, $\theta>1$.
Тогда $|\varphi^k(a)|=1+\sum_{i=1}^k|\varphi ^{i-1}(u)|<C(u)k^d\theta^k$ согласно лемме \ref{rpart}. То есть $(D,\Theta)=(d,\theta)$. Противоречие.
\end{proof}

\begin{lemma}
Если $W$ равномерно рекуррентно, то все буквы из $A$ являются быстрорастущими.
\end{lemma}
\begin{proof}
Если в $A$ не все буквы являются быстрорастущими, то в $\varphi^{\infty}(a_1)$ найдётся подслово вида $e_0C_0f_0$, где $e_0$ и $f_0$ -- быстрорастущие буквы,
а $r(C_0)<(D,\Theta)$.
Для каждого $k\in \mathbb N$ представим $\varphi ^k(e_0)=V'_ke_kV_k$, где $r(e_k)=(D,\Theta)$ и $r(V_k)<(D,\Theta)$. 
Последовательность $\{e_k\}$ заключительно периодична с неким периодом $T_1$.

Аналогично $\varphi ^k(f_0)=U_kf_kU'_k$, $r(U_k)<(D,\Theta)$, $r(f_k)=(D,\Theta)$, посдедовательность $\{f_k\}$ заключительно периодична с периодом $T_2$.

Пусть $\sigma =\varphi ^{T_1T_2}$. Тогда существуют такие буквы $e,f$ и слово $C$, что $\sigma (e)=V'eV$, $\sigma (f)=UfU'$,
$r(e)=r(f)=(D,\Theta)$, $r(U)<(D,\Theta)$, $r(V)<(D,\Theta)$, $r(C)<(D,\Theta)$, $eCf$ -- подслово сверхслова $\varphi^{\infty}(a_1)$.

Введём обозначения: $E_k=\varphi ^{k-1}(V')\varphi ^{k-2}(V')\dots V'e$; $F_k=fU'\varphi (U)\varphi ^2(U)\dots\varphi ^{k-1}(U)$;
$C_k=V\varphi (V)\dots\varphi ^{k-1}(V)\varphi ^k(C)\varphi ^{k-1}(U)\varphi ^{k-2}(U)\dots U$.

Для любого $k$, $E_kC_kF_k$ является подсловом $\varphi ^{\infty}(a_1)$. Каждое следующее $E_k$ оканчивается на предыдущее,
каждое следующее $F_k$ начинается на предыдущее. Из леммы \ref{rpart} следует, что 
$\lim_{k\to \infty}\dfrac{|E_k|}{|C_k|}=\lim_{k\to \infty}\dfrac{|F_k|}{|C_k|}=\infty$. Так как морфизм $\psi$ является нестирающим, то 
$\lim_{k\to \infty}\dfrac{|\psi(E_k)|}{|\psi(C_k)|}=\lim_{k\to \infty}\dfrac{|\psi(F_k)|}{|\psi(C_k)|}=\infty$.

Предположим, что $W=\psi(\varphi^{\infty}(a_1))$ является равномерно рекуррентным. Тогда, как было показано в конце предыдущего раздела, множество 
его конечных подслов совпадает с множеством подслов -- непериодичного сверхслова $H$, порождённого примитивной подстановкой.
Следовательно, для $W$ выполнены следующие свойства:
\begin{proposition}\label{qqq}
Существуют такие положительные действительные $P$ и $K$, что
\begin{enumerate}
\item Если $u_1$ и $u_2$ -- подслова $W$ и $|u_2|\geq P|u_1|$, то $u_1\sqsubseteq u_2$.
\item Если $u\sqsubseteq W$, то то для любых двух различных вхождений $u$ в $W$ их левые концы находятся на расстоянии не меньшем, чем $K|u|$ 
(подразумевается, что $K<1$.)
\end{enumerate}
\end{proposition}
Возьмём такое натуральное $m$, что $\dfrac{3P}m<\dfrac K2$. Найдётся такое $k$, что $|\psi(E_k)|>m|\psi(C_{k+m})|$ и $|\psi(F_k)|>m|\psi(C_{k+m})|$.

В $W$ есть подслова $\psi(E_k)\psi(C_{k+i})\psi(F_k)$ для всех $i=0,1,\dots,m$. 

Пусть $N=\min\{|\psi(E_k)|,|\psi(F_k)|\}$. Таким образом, существует такой набор слов $A$, $B$, $C_i$ для $i=0\dots m$, что
\begin{enumerate}
\item $|A|=|B|=N$;
\item $|C_0|<|C_1|<\dots<|C_m|<\dfrac{N}{m}$;
\item $AC_iB\sqsubseteq W$ для любого $i$.
\end{enumerate}

Все слова $AC_iB$, согласно \ref{qqq}, можно поместить в $U$  -- подслово сверхслова $W$ длины $3PN$.
По принципу Дирихле, для каких-то $i$ и $j$ левые концы слов $AC_iB$ и $AC_jB$ будут находиться в $U$ на расстоянии, не большем $\dfrac{3PN}{m}$.
Тогда правые концы этих слов будут находиться на расстоянии, не большем, чем $\dfrac Nm+\dfrac{3PN}m<KN$. Слово $B$ встретилось со сдвигом, 
меньшим, чем $KN$. Противоречие.
\end{proof}
\begin{remark}
В работе \cite{Dur} доказана примерно такая же лемма.
\end{remark}

Если у всех букв алфавита одинаковый порядок роста, то $D=0$. 
Существуют константы $C_1$ и $C_2$ такие, что для любого $k\in \mathbb N$ и
любой буквы $a_i$ выполнено $C_1\Theta^k<|\psi(\varphi^k(a_i))|<C_2\Theta^k$.

\begin{proposition}
Алгоритмически можно определить, одинаковые ли порядки роста у букв алфавита. Если одинаковые, то можно явно указать $C_1$ и $C_2$, а также оценить
сверху и снизу число $\Theta$.
\end{proposition}
\begin{proof}
Сначала проверим, одинаковые ли порядки роста у букв. Напомним, что буква называется рекуррентной, если она принадлежит некоторой итерации $\varphi$
от самой себя.

\begin{lemma}\label{sq}
Порядки роста у всех букв одинаковые если и только если у всех рекуррентных букв порядки роста одинаковые.
\end{lemma}
\begin{proof}
В одну сторону это очевидно, докажем в другую, показав, что для любой буквы $a$ есть рекуррентная буква с таким же порядком роста.

Рассмотрим последовательность слов $\{u_i\}$, где $u_1=a$, а слово $u_{i+1}$ получается из $\varphi(u_{i})$ вычёркиванием всех рекуррентных букв.
Если какое-то слово $u_k=\varepsilon$, то, пользуясь \ref{qs}, несложно показать, что порядок роста каждого из $u_i$ является порядком роста некоторой
рекуррентной буквы. С другой стороны, для каждой буквы в $u_{i+1}$ в $u_i$ есть та буква, из которой она непосредственно возникла, поэтому, если $|u_{|A|+1}|>0$, то среди букв слов $u_i$ по принципу Дирихле есть рекуррентная буква.    
\end{proof}

Очевидно, что при замене $\varphi$ на $\varphi^k$ одинаковые порядки роста останутся одинаковыми,
а различные -- различными. Рассмотрим $\rho $ -- такую степень $\varphi$, что если $a_i$ -- рекуррентная буква, то $a_i\sqsubseteq \rho(a_i)$.
Для каждой рекуррентной буквы $a_i$ рассмотрим ограничение $\varphi$ на то множество букв, 
которое можно получить из $a_i$ итерациями подстановки $\rho$.
Для $a_i$ рассмотрим на этом множестве букв матрицу подстановки $M_i$, пусть букве $a_i$ в матрице отвечают первая строка и столбец.

Порядок роста $r(a_i)$ -- это порядок роста вектора $M_i^k(a_i)$ в $l_1$-норме. Рассмотрим жорданов базис для оператора $M_i$ 
(возможно, придётся комплексифицировать пространство). Рассмотрим те жордановы клетки, у которых модуль числа на диагонали наибольший. 
Пусть этот модуль -- это $\lambda$, из природы оператора $M_i$ следует, что $\lambda>1$. 
Среди этих клеток возьмём клетку самого большого размера $d$. Тогда скорость роста любого вектора в эрмитовой норме не превосходит 
$Cn^{d-1}|\lambda|^n$, причём вектора, у которых порядок роста меньше, лежат в подпространстве коразмерности $1$. 
Так как все нормы в конечномерном пространстве эквивалентны, то в исходном (некомпактифицированном) пространстве есть вектор с положительными 
координатами и таким же порядком роста нормы. Заметим, что у вектора, соответствующего букве $a_i$, порядок роста не меньше (ибо он не меньше, чем 
порядки роста всех остальных букв, соответствующих строкам матрицы).
Итак, мы выяснили, что если $r(a_i)=(d,\theta)$, то $\theta$ -- это максимальный из модулей собственных значений оператора $M_i$, а $d$ -- размер 
жордановой клетки.

Так как комплексные числа можно представить в виде матриц $2\times 2$ из действительных чисел, то в сигнатуре $\{=,<,0,1,+,\times\}$
можно выразить свойство ``жордановы клетки оператора $M_i$ имеют такой-то вид; клетки оператора $M_j$ имеют такой-то вид;
наборы модулей собственных значений, соответствующих клеткам, упорядочены так-то''.

Из теоремы Тарского-Зайденберга следует, что равенство $r(a_i)=r(a_j)$ алгоритмически проверяемо.

Пусть уже установлено, что все буквы имеют одинаковый порядок роста. В таком случае, этот порядок -- $(0,\theta)$ для некоторого $\theta$.

Из доказательства леммы \ref{sq} следует, что достаточно уметь указывать ограничивающие константы для каждой рекуррентной буквы. 
Кроме того, можно указать константы для подстановки $\rho$. 

Итак, мы снова работаем с матрицей $M_i$. Ограничивающие константы $C_1$ и $C_2$ можно вывести, если известны:
\begin{enumerate}
\item вид жордановой формы (то есть упорядочивание жордановых клеток по модулям коэффициентов и их размеры);
\item информация про коэффициенты разложения вектора $(1,0,\dots,0)$ по некоторому жорданову базису (коэффициент при одном из базисных векторов, отвечающий наибольшему собственному значению, ненулевой. Нам достаточно оценки снизу на его отношение к остальным коэффициентам).
\item оценки сверху на число обусловленности матрицы перехода к данному жорданову базису. 
\end{enumerate}

Первый пункт делается применением теоремы Тарского-Зайденберга. Покажем, как можно получить оценки во втором и третьем пунктах.
Из алгоритма поиска жорданова базиса следует, что существует матрица перехода к жорданову базису, коэффициенты которой являются
рациональными функциями от корней характеристического многочлена матрицы $M_i$. Значит, для второго и третьего пунктов существуют оценки, являющиеся 
отношениями двух многочленов от корней характеристического многочлена, при этом эти оценки не равны нулю. В таком случае можно как угодно оценить сверху
значение знаменателя, а значение числителя оценить снизу с помощью обобщённой теоремы Лиувилля.
\end{proof}

Далее можно считать, что порядки роста всех букв равны. Тогда, согласно \ref{Pr}, задачу можно свести к следующей:

{\bf Дано:} алфавиты $A$, $B$, $C$; морфизмы $\varphi:C^*\to C^*$, $\psi :C^*\to B^*$, $g:A^*\to A^*$, $h:A^*\to B^*$, а также числа $C'_1$, $C'_2$, $\Theta_1>1$, $\Theta_2$.
При этом все морфизмы нестирающие, морфизм $\varphi$ примитивен и продолжается над $c_1$ и известно, что $C'_1\lambda^k<|h(g^k(c_i))|<C'_2\lambda^k$ для всех $k$ и некоторого 
$\lambda\in [\Theta_1;\Theta_2]$. 

{\bf Определить:} верно ли что для всех натуральных $k$ и букв $a_i\in A$ каждое из слов $h(g^k(a_i))$ является подсловом сверхслова
 $W=\psi(\varphi^{\infty}(c_1))$?
 
 \begin{theorem}\label{nos}
 Эта задача алгоритмически разрешима.
 \end{theorem}
Для доказательства потребуется язык схем Рози.
\section{Схемы Рози.}

Напомним основные определения и факты, связанные со схемами Рози и описанные в работах \cite{MKR,MP}.

{\it Графом со словами} будем называть сильносвязный ориентированный граф, у которого на каждом ребре написано по два слова -- {\it переднее} и {\it заднее}. Также потребуем, чтобы каждая вершина либо имеет входящую степень $1$, а исходящую больше $1$, либо входящую степень больше $1$ и исходящую степень $1$. Вершины первого типа назовём {\it раздающими}, а второго -- {\it собирающими}.

{\it Путь} в графе со словами -- это последовательность рёбер, каждое следующее из которых выходит из той вершины, в которую входит предыдущая. {\it Симметричный путь} -- это путь, первое ребро которого начинается в собирающей вершине, а последнее ребро кончается в раздающей.

Каждый путь можно записать словом над алфавитом -- множеством рёбер графа, и это слово называется {\it рёберной записью пути}. Для двух путей выполняются отношения {\it подпути} (пишем $s_1\sqsubseteq s_2$), {\it начала} или {\it конца}, если для их рёберных записей выполняются соответственно
отношения подслова, начала или конца.
Кроме того, пишем $s_1\sqsubseteq_k s_2$, если для соответствующих слов $u_1$ и $u_2$ -- рёберных записей путей $s_1$ и $s_2$ --  выполнено $u_1\sqsubseteq_k u_2$.
Если последнее ребро пути $s_1$ идёт в ту же вершину, из которой выходит первое ребро пути $s_2$, путь, рёберная запись которого является конкатенацией рёберных записей путей $s_1$ и $s_2$, будем обозначать $s_1s_2$.

Для каждого пути $s$ в графе со словами определим {\it переднее слово} $F(s)$. Пусть $v_1v_2\dots v_n$ -- рёберная запись пути $s$. В $v_1v_2\dots v_n$ возьмём подпоследовательность: включим в неё $v_1$, а также те и только те рёбра, которые выходят из раздающих вершин графа. Эти рёбра назовём {\it передними образующими для пути $s$}. Возьмём передние слова этих рёбер и запишем их последовательную конкатенацию, это и будет $F(s)$.
\begin{remark}
Запись $|s|$ обозначает длину пути, которая измеряется в рёбрах. $|F(s)|$ -- это длина слова.
\end{remark}

Аналогично определяется $B(s)$. В $v_1v_2\dots v_n$ возьмём рёбра, входящие в собирающие вершины и ребро $v_n$ в порядке следования -- это {\it задние образующие для пути $s$}. Тогда последовательной конкатенацией задних слов этих рёбер получается
{\it заднее слово $B(s)$ пути $s$}.

\begin{definition}
Если $S$ – сильносвязный граф, не являющийся циклом, и $s$
– путь в графе, то естественное продолжение пути $s$ вправо – это минимальный
путь, началом которого является $s$ и который оканчивается в раздающей вершине.
Естественное продолжение пути $s$ влево – это минимальный путь, концом которого
является $s$ и который начинается в собирающей вершине.
\end{definition}

\begin{definition} \label{Def1}
Граф со словами будет являться {\it схемой Рози} для рекуррентного сверхслова $W$, если он удовлетворяет следующим свойствам, которые в дальшейшем будут называться {\it свойствами схем Рози}:

\begin{enumerate}
    \item Граф сильносвязен и состоит более чем из одного ребра.
    \item Все рёбра, исходящие из одной раздающей вершины графа, имеют передние слова с попарно разными первыми буквами.
	Все рёбра, входящие в одну собирающую вершину графа, имеют задние слова	с попарно разными последними буквами.
    \item Для любого симметричного пути, его переднее и заднее слова совпадают. То есть можно говорить просто о слове симметричного пути.
    \item Если есть два симметричных пути $s_1$ и $s_2$ и выполнено $F(s_1)\sqsubseteq _k F(s_2)$, то $s_1\sqsubseteq _k s_2$.
	\item Все слова, написанные на рёбрах графа, являются подсловами $W$.
	\item Для любого $u$~-- подслова $W$ существует симметричный путь, слово которого содержит $u$.  
	\item Для любого ребра $s$ существует такое слово $u_s$, принадлежащее $W$, что любой симметричный путь, слово которого содержит $u_s$, проходит по ребру $s$.
\end{enumerate}  
\end{definition}
{\it Опорным ребром} в схеме называется любое ребро, входящее в раздающую вершину и выходящее из собирающей. В любой схеме Рози присутствует хотя бы одно опорное ребро, и у каждого опорного ребра переднее слово равно заднему.
{\it Масштабом} схемы Рози называется наименьшая из длин слов на опорных рёбрах. В \cite{MP} показано, как для непериодичного рекуррентного сверхслова получать с помощью {\it графов Рози} схему Рози сколь угодно большого масштаба.

Определим {\it эволюцию} $(W,S,v)$ схемы Рози $S$ по опорному ребру $v$: 
пусть $\{x_i\}$ -- множество рёбер. входящих в начало $v$, а $\{y_i\}$ -- множество рёбер, идущих из конца $v$ (эти два множества могут пересекаться).
Обозначим $F(y_i)=Y_i$, $B(x_i)=X_i$, $F(v)=V$.
Рассмотрим все слова вида $X_i V Y_j$. Если такое слово не входит
в $W$, то пару $(x_i,y_j)$ назовём {\it плохой}, в противном случае -- {\it хорошей}. Также {\it хорошей} или {\it плохой} будем называть
 соответствующую тройку рёбер $(x_i,v,y_j)$.

Построим граф $S'$. Он получается из $S$ заменой ребра $v$ на $K_{\#\{x_i\},\#\{y_j\}}$, где $K_{m,n}$ -- полный двудольный граф. Более подробно: ребро $v$ удаляется, его начало заменяется на множество $\{A_i\}$ из $\#\{x_i\}$ вершин так, что для любого $i$ ребро $x_i$ идёт в $A_i$;
конец ребра $v$ заменяется на множество $\{B_j\}$ из $\#\{y_j\}$ вершин так, что для любого $j$ ребро $y_j$ выходит из $B_j$; вводятся рёбра $\{v_{i,j}\}$, соединяющие вершины множества $\{A_i\}$
с вершинами множества $\{B_j\}$.

По сравнению с $S$, у графа $S'$ нет ребра $v$, но есть новые рёбра $\{v_{ij}\}$. Остальные рёбра граф $S$ взаимно однозначно соответствуют рёбрам графа $S'$. Соответственные рёбра в первом и втором графе зачастую будут обозначаться одними и теми же буквами.

На рёбрах $S'$ расставим слова следующим образом. На всех рёбрах $S'$, кроме рёбер из $\{y_i\}$ и $\{v_{i,j}\}$, передние слова пишутся те же, что и передние слова соответственных рёбер в $S$. Для каждого $i$ и $j$, переднее слово ребра $y_j$ в $S'$ -- это $VY_j$. Переднее слово ребра $v_{i,j}$ -- это $Y_j$.
Аналогично, на всех рёбрах, кроме рёбер из $\{x_i\}$ и $\{v_{i,j}\}$, задние слова переносятся с соответствующих рёбер $S$; для всех $i$ и $j$ в качестве заднего слова ребра $x_i$ возьмём $X_iV$, а в качестве заднего слова ребра $v_{i,j}$ -- $X_i$.

Теперь построим граф $S''$. Рёбра $v_{i,j}$, соответствующие плохим парам $(x_i,y_j)$, назовём {\it плохими}, 
а все остальные рёбра графа $S'$ -- {\it хорошими}. Граф $S''$ получается, грубо говоря, удалением плохих рёбер из $S'$.
Более точно, в графе $S''$ раздающие вершины -- подмножество раздающих вершин $S'$, а собирающие вершины -- подмножество собирающих вершин $S'$.
В графе $S'$ эти подмножества -- вершины, из которых выходит более одного хорошего ребра, и вершины, в которые входит более одного хорошего ребра соответственно.
Назовём в $S'$ вершины этих двух подмножеств $S'$ неисчезающими. Рёбра в графе $S''$ соответствуют таким путям в $S'$, которые идут лишь по хорошим рёбрам, начинаются в неисчезающих вершинах, заканчиваются в неисчезающих, а все промежуточные вершины которых не являются неисчезающими.

Несложно показать, что $S''$ -- сильносвязный граф, не являющийся циклом. У каждого ребра в $S''$ есть естественное продолжение вперёд и назад. Естественное продолжение вперёд соответствует пути в $S'$; переднее слово этого пути в $S'$ возьмём в качестве переднего слова для соответствующего ребра $S''$. Аналогично для задних слов: у ребра в $S''$ есть естественное продолжение влево, этому пути в $S''$ соответствует путь в $S'$. Заднее слово этого пути и будет задним словом ребра в $S''$.

\begin{definition} Построенный таким образом граф со словами $S''$ назовём {\it элементарной эволюцией} $(W,S,v)$. 
\end{definition}

\begin{theorem} \label{T1}
Элементарная эволюция $(W,S,v)$ является схемой Рози для сверхслова $W$ (то есть удовлетворяет свойствам $1$---$7$ определения \ref{Def1}).
\end{theorem}

\begin{definition}
Пусть $S$ -- схема Рози для сверхслова $W$. На её рёбрах можно написать различные натуральные числа (или пары чисел). Такую схему мы назовём {\it нумерованной}. Если с рёбер пронумерованной схемы стереть слова, получится {\it облегчённая нумерованная схема}.
\end{definition}

\begin{definition}
{\it Метод эволюции} -- это функция, которая каждой облегчённой пронумерованной схеме (с нумерацией, допускающей двойные индексы) даёт этой же схеме новую нумерацию, такую, что в ней используются числа от $1$ до $n$ для некоторого $n$ по одному разу каждое.
\end{definition}

Зафиксируем какой-либо метод эволюции и далее не будем его менять.

Среди опорных рёбер пронумерованной схемы $S$ возьмём ребро $v$ с наименьшим номером, и совершим элементарную эволюцию $(W,S,v)$.
Укажем естественную нумерацию новой схемы.

Напомним, что сначала строится схема $S'$, а потом -- $S''$. В схеме $S'$ все рёбра можно пронумеровать по следующему правилу: рёбра кроме $v$ сохраняют номера, а рёбра вида $v_{ij}$ нумеруют соответствующим двойным индексом.

Каждое ребро в схеме $S''$  -- это некоторый путь по рёбрам схемы $S'$, различным рёбрам из $S''$ соответствуют в $S'$ пути с попарно различными первыми рёбрами. Таким образом, рёбра схемы $S''$ можно пронумеровать номерами первых рёбер соответствующих путей $S'$. Теперь применим к облегчённой нумерованной схеме $S''$ метод эволюции. Получится новая облегчённая нумерованная схема, и в нумерованной (не облегчённой) схеме $S''$ перенумеруем рёбра соответственным образом.

\begin{definition}
Описанное выше соответствие, ставящее нумерованной схеме Рози другую нумерованную схему Рози, назовём {\it детерменированной эволюцией}. Будем обозначать это соответствие $S''=\Evol(S).$
\end{definition}

\begin{definition}
Применяя детерменированную эволюцию к нумерованной схеме $S$ много раз, получаем последовательность нумерованных схем Рози.
Кроме того, на каждом шаге получаем множество пар чисел, задающие плохие пары рёбер. {\it Протокол детермеминованной эволюции} -- это последовательность таких множеств пар чисел и облегчённых нумерованных схем.
\end{definition}

Из определения элементарной эволюции следует
\begin{proposition}
Облегчённая нумерованная схема для $\Evol(S)$ однозначно определяется по облегчённой нумерованной схеме $S$ и множеству пар чисел,
задающие плохие пары рёбер.
\end{proposition}
\begin{proposition}\label{f}
Каждому симметричному пути в $Evol(S)$ соответствует симметричный путь в $S$ с таким же словом; длина соответствующего пути в $S$ не меньше, чем в $Evol(S)$ (строится соответствие так: $Evol(S)=S''$ естественным образом вкладывается в $S'$, а после применяется отображение путей $f^{-1}$, описанное в \cite{MKR}). При этом соответствии сохраняется отношение 
$\sqsubseteq_k$, а также отношения ``начала пути'' и ``конца пути''. Опорное ребро, по которому делалась эволюция, не соответствует в $Evol(S)$ никакому пути. Для любого симметричного пути в $S$, не
явлюящегося этим опорным ребром, в $S$ существует минимальный симметричный путь, cоответствующий некоторому симметричному пути в $Evol(S)$. Все соответствующие пути можно определить по облегчённым нумерованным схемам.
\end{proposition}

В работе \cite{MKR} доказана следующая теорема:
\begin{theorem}\label{MKR}
Если $W$ -- примитивное подстановочное непериодичное сверхслово, то протокол его детерменированной эволюции периодичен с предпериодом.
\end{theorem}
При этом предпериод и период алгоритмически находятся по морфизмам, порождающим слово $W$.

\begin{definition}
Если $S$ -- схема Рози для сверхслова $W$, а $T$ -- натуральное число, то в соответствующей облегчённой нумерованной схеме можно указать, какие из 
симметричных путей длины не более $T$ являются допустимыми. 
Этот набор (облегчённая нумерованная схема + набор допустимых путей не длиннее $T$) будем называть {\it $T-$разруленной схемой}.
\end{definition}

Из доказательства теоремы \ref{MKR} (см. \cite{MKR}) вытекает следующее предложение:
\begin{proposition}\label{MKRT}
Если $W$ -- примитивное подстановочное сверхслово, а $T$ -- произвольное число, то для протокола детерменированной эволюции последовательность
соответствующих $T-$разруленных схем периодична с предпериодом. Опять же, предпериод и период алгоритмически находятся.
\end{proposition}

\section{Построение алгоритма для теоремы \ref{nos}.}

\begin{proposition}\label{big}
 Пусть $W=\psi(\varphi ^{\infty}(c_1))$ -- непериодичное примитивное сверхслово.
\begin{enumerate}
\item Можно явно найти такое число $P$, что если $u_1$ и $u_2$ -- подслова $W$ и $|u_2|\geq P|u_1|$, то $u_1\sqsubseteq u_2$.
\item Можно явно найти такое число $C$, что если $u\sqsubseteq W$, то то для любых двух различных вхождений $u$ в $W$
их левые концы находятся на расстоянии не меньшем, чем $C|u|.$ 
\item Можно явно найти такое число $C_{max}$, что в любой схеме Рози $S$ длина слова любого допустимого пути не менее $C_{min}Mn$, где $n$ -- количество рёбер в пути, а $M$ -- масштаб схемы $S$.
\item Можно явно найти такое число $C_{min}$, что в любой схеме Рози $S$ длина любого переднего и заднего слова на рёбрах схемы не превосходит $C_{max}M$, где $M$ -- масштаб схемы $S$. Следовательно, если путь содержит $n$ рёбер, то длина переднего и заднего слов этого пути не превосходит
$C_{max}Mn$.
\item Масштаб схемы $Evol(S)$ не меньше масштаба схемы $S$. Можно явно найти такое $C_m$, что масштабы схем $Evol(S)$ и $S$ относятся не более, чем в
$C_m$ раз для любой схемы $S$.
\end{enumerate}
\end{proposition}
Доказательства всех утверждений предложения \ref{big} можно найти в \cite{MKR}.

Далее слово $W$, а также морфизмы $\varphi$ и $\psi$ считаем неизменными.

Пусть $s$ -- симметричный путь по рёбрам схемы $S$. Пусть последовательность его рёбер -- это $v_1,v_2,\dots,v_n$.
Предположим, что $v_{i_1},v_{i_2},\dots,v_{i_m}$ -- его передние образующие.
Если длина слова $F(s)$ превосходит $C_{max}M$, то передних образующих рёбер более одного. Путь $v_1v_2\dots v_{i_m-1}$ также является симметричным. 
Будем обозначать его $s_{right}$ и по отношению к $s$ называть {\it урезанным справа.} Так как $F(s)=F(s_{right})F(v_{i_m})$, то $F(s_{right})\geq F(s)-C_{max}M$.

Аналогично определяется путь $s_{left}$, который является по отношению к $s$ {\it урезанным слева.}

Если $F(s)\geq 2C_{max}M$, то существует путь $(s_{left})_{right}=(s_{right})_{left}$, который будем обозначать $\cut(s)$. Очевидно, $F(s)=B(v_1)\cut(s)F(v_2)$
для некоторых рёбер $v_1$ и $v_2$.

Пусть $S$ -- схема Рози для сверхслова $W$, $A$ -- слово. Если $A\sqsubseteq W$, то $A\sqsubseteq F(s)$ для некоторого допустимого пути $s$.
\begin{definition}
Любой минимальный по включению путь $s$ будем обозначать $l(S,A)$. Здесь $S$ -- схема, $A$ -- слово. Далее будет показано, что во многих случаях
такой путь единственный.
\end{definition}

\begin{lemma} \label{lmp}
Длина слова пути $l(S,A)$ не превышает $|A|+2C_{max}M$.  
\end{lemma}
\begin{proof}
Пусть $s=l(S,A)$. Слово $F(s)$ можно представить в виде конкатенации $B(v_1)F(\cut(s))F(v_2)$. В этом слове содержится $A$, но при этом $A$
из-за минимальности $s$ не содержится в $B(v_1)F(\cut(s))=F(s_{right})$ и в $F(\cut(s))F(v_2)=F(s_{left})$. Следовательно, $A=u_1F(\cut(s))u_2$
для некоторых непустых $u_1$ и $u_2$. А как мы знаем, $F(\cut(s))\geq F(s)-2C_{max}M$.  
\end{proof}

\begin{lemma} \label{emp}
Существует такое число $K$, что если $M$ -- масштаб схемы $S$, а длина $|A|$ не меньше $KM$, то $l(S,A)$ единственен. 
\end{lemma}
\begin{proof}
Зафиксируем один минимальный путь $s$. Также рассмотрим произвольный минимальный путь $s'$.
Из доказательства леммы \ref{lmp} следует, что $A=u_1F(\cut(s))u_2$ для некоторых непустых $u_1$ и $u_2$.
Так как $A\sqsubseteq F(s')$, то $F(\cut(s))\sqsubseteq F(s')$. По свойству $4$ схем Рози, $\cut(s)\sqsubseteq s'$. Длины слов $F(\cut(s))$ и $F(s')$
отличаются не более, чем на $4MC_{max}$, следовательно, если $\frac{4MC_{max}}{|A|-2MC{max}}<C$, то у слова $F(\cut(s))$ ровно одно вхождение в $F(s')$.
Выполнения этого неравенства легко достичь, выбирая достаточно большое $K$.

Пусть $v_1v_2\dots v_n$ -- ребра пути $s'$, при этом рёбра $v_k,v_{k+1},\dots v_l$ образуют путь $\cut(s)$.
Тогда $B(v_{k-1})F(\cut(s))F(v_{l+1})\sqsubseteq F(s')$. Так как у $F(\cut(s))$ ровно одно вхождение в $F(s')$, то $F(v_{l+1})$ и $u_2$ начинаются на одну и ту же букву, а $B(v_{k-1})$ и $u_1$ кончаются на одну и ту же букву. Следовательно, у путей $s$ и $s'$ первое ребро после вхождения $\cut(s)$ одно и то же (а именно то, переднее слово которого начинается с первой буквы слова $u_1$.) Аналогично, ребро непоследственно перед вхождением $\cut(s)$
одно и то же. Следовательно, $s'$ содержит путь $s$. Из минимальности следует, что $s=s'$.
\end{proof}

\begin{proposition}\label{opa}
Если $S$ -- схема с масштабом $M$, $s$ -- допустимый путь, $|A|\geq KM$, $A\sqsubseteq s$, $A=u_1F(\cut(s))u_2$ для непустых $u_1$, $u_2$. Тогда
$l(S,A)=s$. 
\end{proposition}

\begin{proof}
Пусть $A\sqsubseteq F(s_{right})$. Тогда у $F(\cut(s))$ есть вхождение в $A\sqsubseteq F(right)$, не являющееся его концом. Но у слова $F(\cut(s))$
есть ровно одно вхождение в $F(s)$ (см. доказательство леммы \ref{emp}). Противоречие.

Аналогично доказывается, что $F(s_{left})$ не содержит $A$. Следовательно, $s=l(S,A)$.
\end{proof}

\begin{lemma} \label{ututu}
Пусть $S$ -- схема масштаба $M$, а длины слов $A$ и $B$ не менее $KM$. Тогда путь $l(S,A)$ -- начало пути $l(S,AB)$.
\end{lemma}
\begin{proof}
Пусть $v_1v_2\dots v_n$ -- последовательность рёбер $l(S,AB)$. Для некоторого $k$, $v_{k+1}v_{k+2}\dots v_n$ -- последовательность рёбер
 пути $l(S,AB)_{left}$. Тогда $AB=u_1u_2$, $B(v_k)$ кончается на $u_1$ и $F(v_{k+1}v_{k+2}\dots v_n)$ начинается на $u_2$.
Так как $|u_1|\leq C_{max}M < |A|$, то $A=u_1u_3$. Слово $F(l(S,AB)_{left})$ начинается на $u_3$, следовательно, существует такое минимальное $k'$, что
$v_{k+1}v_{k+2}\dots v_{k'}$ -- симметричный путь, слово которого начинается с $u_3$.
Докажем, что $l(S,A)=v_1v_2\dots v_{k'}$. В самом деле,
$$A=u_1u_3\sqsubseteq B(v_k)F(v_{k+1}v_{k+2}\dots v_{k'})=F(v_1v_2\dots v_{k'}).$$
Из минимальности $k'$ следует, что $u_3=F((v_{k+1}v_{k+2}\dots v_{k'})_{right})u_4$ для некоторого непустого $u_4$. Следовательно,
$A=u_1F(\cut (v_1v_2\dots v_{k'}))u_4$. Утверждение леммы следует из предложения \ref{opa}. 
\end{proof}

Из доказательства этой леммы следует следующее предложение:
\begin{proposition}\label{zc}
Пусть $S$ -- схема масштаба $M$, а длины слов $A$ и $B$ не менее $KM$. Если $F(l(S,A))=u_1Au_2$, а $F(l(S,AB))=u_3ABu_4$. Тогда $u_1=u_3$.
\end{proposition}

\begin{lemma}\label{nya}
Пусть $S$ -- схема Рози сасштаба $M$, слова $A$,$B$,$C$ имеют длину более $KM$, слова $AB$ и $BC$ являются подсловами $W$.
Тогда следующие условия эквивалентны:
\begin{enumerate}
\item Слово $ABC$ является подсловом $W$.
\item Склейка путей $l(S,AB)$ и $l(S,BC)$ по пути $l(S,B)$ является допустимым путём (согласно лемме \ref{ututu}, путь $l(S,B)$ является концом $l(S,AB)$ и началом $l(S,BC)$).
Иначе говоря, путь $l$, имеющий длину $|l(S,AB)|+|l(S,BC)|-|l(S,B)|$ и такой, что $l(S,AB)$ является началом пути $l$, а $l(S,BC)$ -- концом $l$, является
допустимым.
\end{enumerate}
При этом путём $l(S,ABC)$ будет являться построенный путь $l$.
\end{lemma}
\begin{proof}
Предположим, что $ABC\sqsubseteq W$. Рассмотрим путь $s=l(S,ABC)$. Согласно лемме \ref{ututu}, $l(S,AB)$ является началом $s$, а $l(S,BC)$ -- 
концом $s$. Пусть $F(ABC)$=$u_1ABCu_2$. Тогда, согласно \ref{zc}, $F(l(S,AB))=u_1ABu_3$, а $F(l(S,BC))=u_4BCu_2$. Также путь $l(S,B)$ является концом
пути $l(S,AB)$ и началом пути $l(S,BC)$, а слово $F(l(S,B))$ представляется в виде $u_4Bu_3$.

Пусть $v_1v_2\dots v_n$ -- последовательность рёбер $s$. Тогда найдутся такие $k_1$ и $k_2$, что $l(S,AB)=v_1v_2\dots v_{k_1}$ и 
$l(S,BC)=v_{k_2}v_{k_2+1}\dots v_n$.
Предположим, что $k_1<k_2$. Тогда рассмотрим минимальный симметричный путь $s'$, последнее ребро которого -- $v_{k_1}$. Так как $B(s')=B(v_{k_1})$, то
$F(s')\leq C_{max}M$. С другой стороны, 
$$F(s')+F(s)=F(v_1v_2\dots v_{k_1})+(F(s')+F(v_{k_1}v_{k_1+1}\dots v_n))>F(l(S,AB))+F(l(S,BC)).$$
Но правая часть превосходит левую хотя бы на $|B|-C_{max}M$. Противоречие.
Таким образом, $k_1\leq k_2$ и мы можем рассмотреть симметричный путь $v_{k_2}v_{k_2+1}\dots v_{k_1}$. Длина слова этого пути
$$|F(v_{k_2}v_{k_2+1}\dots v_{k_1})|=|F(l(S,AB))|+|F(l(S,BC))|-|F(s)|=|u_3|+|B|+|u_4|.$$
У слова $F(l(S,B))$ такая же длина и оно также является окончанием слова $F(l(S,AB))$. Следовательно, $F(v_{k_2}v_{k_2+1}\dots v_{k_1})=F(l(S,B))$ 
и $v_{k_2}v_{k_2+1}\dots v_{k_1}=l(S,B)$ по $4$-му свойству схем Рози. Значит, допустимый путь $s$ -- это путь $l$, фигурирующий во втором условии.
Доказана импликация $1\to 2$.

Предположим теперь, что выполнено второе условие. Докажем, что $F(l)$ содержит $ABC$. Воспользуемся предложением \ref{zc} и представим слова в 
следующем виде: $F(l(S,AB))=u_1ABu_2$, $F(l(S,BC))=u_3BCu_4$, $F(l(S,B))=u_3Bu_2$.
Пусть $v_1v_2\dots v_n$ -- последовательность рёбер пути $l$, при этом рассматриваемый подпуть $l(S,B)$ -- это $v_{k_1}v_{k_1+1}\dots v_{k_2}$.
Обозначим $D=F(v_{k_2+1}v_{k_2+2}\dots v_n)$. Тогда $u_3BCu_4=u_3Bu_2D$, следовательно, $Cu_4=u_2D$. Но $F(l)=u_1ABu_2D=u_1ABCu_4$. Доказана
импликация $2\to 1$.
\end{proof}

\begin{definition}
Все подслова сверхслова $g^{\infty}(a)$, имеющие длину $1$ или $2$, назовём {\it $g-$источниками}.
\end{definition}

Все $g-$источники, согласно \ref{monad}, находятся алгоритмически.
Пусть их количество равно $m$.
\begin{definition} Слова $h(g^k(q_1)),h(g^k(q_2)),\dots,h(g^k(q_m))$, где $q_i$-- это $g-$источники, назовём {\it рабочими словами порядка $k$.}
\end{definition}

\begin{proposition} \label{lww}
Можно найти такие положительные числа $C_1$ и $C_2$, что для любых $k\in \mathbb N$ и $g-$источника $q_i$ выполнено $C_1\lambda^k<|h(g^k(q_i))|<C_2\lambda^k$.
\end{proposition}

Далее фиксируем число $T$. Оно предполагается достаточно большим; явно укажем его позднее.

\begin{definition}
Пусть $S$ -- нумерованная схема Рози, $k$ -- натуральное число, все рабочие слова порядка $k$ являются подсловами $W$.
Тогда множество, состоящее из
\begin{enumerate}
\item соответствующей облегчённой нумерованной схемы для $S$;
\item множества всех допустимых путей не длиннее $T$, рассматриваемых как пути в облегчённой схеме (эти пути назовём {\it проверочными});
\item номера $T-$разруленной схемы $S$ в периоде или предпериоде (согласно \ref{MKRT}, эти схемы периодичны с предпериодом);
\item множества путей $l(S,h(g^k(q_i)))$ для всех источников $q_i$ (пути рассматриваются в облегчённой схеме, назовём их {\it основными путями}).
\end{enumerate}
назовём {\it антиоснасткой $(S,k)$}.
\end{definition}

\begin{definition}
{\it Размер антиоснастки $(S,k)$} -- это максимальная длина (в рёбрах) по путям из $l(S,h(g^k(q_i)))$ для всех $g-$источников $q_i$.
\end{definition}

\begin{lemma} \label{size}
Можно найти такие положительные $C_3$, $C_4$ и $C_5$, что для любой схемы $S$ и любого $k$ размер антиоснастки $(S,k)$ заключён между
$C_3\frac{\lambda^k}{M}$ и $C_4\frac{\lambda^k}{M}+C_5$, где $M$ -- масштаб схемы.
\end{lemma}
\begin{proof}
Пусть $p_i$ -- рабочее слово порядка $k$. Согласно \ref{lww}, $|p_i|>C_1\lambda^k$.
Если $s=l(S,p_i)$, то $|F(s)|\geq|p_i|$. С другой стороны, $F(s)\leq C_{max}M|s|$. Стало быть, $|s|\geq \frac{C_1\lambda^k}{C_{max}M}$.

С другой стороны, $C_{min}M|s|\leq F(s)\leq |p_i|+2MC_{max}<C_2\lambda^k+2MC_{max}$. 
Следовательно, $|s|<\frac{C_2\lambda^k}{C_{min}M}+2\frac{C_{max}}{C_{min}}$.
\end{proof}

\begin{corollary} \label{wts}
Можно указать явно такие $C_{6}$, $C_{7}$, $C_8$ и $C_9$, что если размер антиоснастки $(S,k)$ больше $C_6$, а $M$ -- масштаб схемы $S$, то
\begin{enumerate}
\item Длины всех рабочих слов составляют не менее $KM$.
\item Размеры антиоснасток $(S,k)$ и $(\Evol(S),k)$ относятся не более, чем в $C_{7}$ раз.
\item Если схема $S_1$ получена из $S$ применением некоторого количества операций $\Evol$, то размер антиоснастки $(S_1,k)$ не превосходит размера 
антиоснастки $(S,k)$, умноженного на $C_7$.
\item Размеры антиоснасток $(S,k)$ и $(S,k+1)$ отличаются не более, чем в $C_{8}$ раз.
\item Размер оснастки $(S,k+C_9)$ больше размера оснастки $(S,k)$ хотя бы в два раза.
\end{enumerate}
\end{corollary}
\begin{proof}
Пусть $x$ -- размер антиоснастки $(S,k)$. В различных пунктах будут различные требования вида ограничения на $x$ снизу. Выберем $C_6$ исходя из самого сильного требования.
\begin{enumerate}
\item Из \ref{size} и \ref{lww} следует, что если $x$ -- размер антиоснастки, а $q$ -- рабочее слово порядка $k$, то
$\lambda^k>\frac{(x-C_5)M}{C_4}$ и $|q|>C_1\lambda^k$. Стало быть, $|q|>\frac{C_1(x-C_5)M}{C_4}$. Достаточно выполнения неравенства $KC_4<C_1(x-C_5)$.
\item Пусть $M_1$ -- масштаб схемы $Evol(S)$, $x$ и $x_1$ -- размеры антиоснасток.
Тогда выполнены двойные неравенства $C_4\frac{\lambda^k}{M_1}+C_5>x_1>C_3\frac{\lambda^k}{M_1}$ и $C_4\frac{\lambda^k}{M}+C_5>x>C_3\frac{\lambda^k}{M}$.

Мы знаем, что $M\leq M_1\leq C_mM$ (см. \ref{big}).
Получаем $x_1>\frac{C_3M(x-C_5)}{M_1C_4}$. Если $x>2C_5$, то $x_1>\frac{C_3}{2C_4C_m}x$.

С другой стороны, $x_1<\frac{C_4xM}{M_1C_3}+C_5\leq 2\frac{C_4}{C_3}x$ при $x>\frac{C_5C_3}{C_4}$.

\item Согласно \ref{big}, масштаб схемы $S_1$ не меньше, чем $M$. Далее см. доказательство предыдущего пункта. 

\item Пусть $x$ и $x_1$ -- размеры антиоснасток $(S,k)$ и $(S,k+1)$.

Тогда $C_3\frac{\lambda^{k+1}}{M}<x_1<C_4\frac{\lambda^{k+1}}{M}+C_5$. С другой стороны, 
$\frac{x-C_5}{C_4}<\frac{\lambda^{k}}{M}<\frac{x}{C_3}$.

Таким образом, $\frac{C_3\lambda(x-C_5)}{C_4}<x_1<\frac{C_4\lambda x}{C_3}+C_5$. При $x>2C_5$ и $x>\frac{C_5C_3}{C_4}$ получаем, что 
$$\frac{C_3\lambda}{2C_4}<\frac{x_1}{x}<\frac{2C_4\lambda}{C_3}.$$
Заметим, что нам не обязательно знать $\lambda$ точно. Достаточно знать какую-нибудь оценку сверху.

\item Пусть $x$ и $x_1$ -- размеры антиоснасток $(S,k)$ и $(S,k+C_9)$.

Тогда $C_3\frac{\lambda^{k+C_9}}{M}<x_1$. С другой стороны, $\frac{\lambda^{k}}{M}>\frac{x-C_5}{C_4}$.

Таким образом, при $x$ > $2C_5$ выполняется $x_1>\frac{C_3\lambda^{C_9}}{2C_4}x$, что не меньше $2x$ при достаточно большом $C_9$.

Заметим, что нам не обязательно знать $\lambda$ точно. Достаточно знать какую-нибудь оценку снизу, отделяющую $\lambda$ от $1$.
\end{enumerate}
\end{proof}

\begin{lemma}\label{hare}
Если размер антиоснастки $(S,k)$ не меньше $C_6$, то по антиоснастке $(S,k)$ можно алгоритмически определить, верно ли, что все рабочие слова порядка
$k+1$ являются подсловами $W$, и, если верно, то алгоритмически найти антиоснастку $(S,k+1)$.
\end{lemma}
\begin{proof}
Пусть $x$ -- размер антиоснастки $(S,k)$.
Необходимо для каждого источника $p_i$ определить, является ли $h(g^{k+1}(p_i))$ подсловом $W$, и если является, то найти путь $l(S,h(g^{k+1}(p_i)))$.

Пусть $g(p_i)=a_{i_1}a_{i_2}\dots a_{i_m}$.  
$$
h(g^{k+1}(p_i))=h(g^k(a_{i_1}))h(g^k(a_{i_2}))\dots h(g^k(a_{i_m})).
$$ 

Каждая из букв $a_{i_1}$, $a_{i_2}$, \dots, $a_{i_m}$ является источником.
Также источниками являются $a_{i_1}a_{i_2}$, $a_{i_2}a_{i_3}$,\dots, $a_{i_{m-1}}a_{i_m}$.

Согласно лемме \ref{wts}, длины слов $h(g^k(a_{i_1}))$, $h(g^k(a_{i_2}))$, $h(g^k(a_{i_3}))$ не менее $KM$,
поэтому для схемы $S$, слов $A=h(g^k(a_{i_1}))$, $B=h(g^k(a_{i_2}))$ и $C=h(g^k(a_{i_3}))$ выполняются условия леммы \ref{nya}.

Из антиоснастки $(S,k)$ нам известны пути $l(S,AB)$, $l(S,BC)$ и $l(B)$. Чтобы определить, является ли $ABC$ подсловом сверхслова $W$,
нужно узнать, является ли некоторый путь допустимым. Длина этого пути равна $|l(S,AB)|+|l(S,BC)|-|l(S,B)|$ и не превосходит $2x$.
Если $T>2x$, то определить допустимость пути мы можем, проверив, лежит ли он среди проверочных путей. Если он допустимый, то он и является путём $l(ABC)$.

Далее воспользуемся леммой \ref{nya}, применённой к словам 
$h(g^k(a_{i_1}a_{i_2}))$, $h(g^k(a_{i_3}))$, $h(g^k(a_{i_4}))$ и аналогичным способом определим путь
$l(S,h(g^k(a_{i_1}a_{i_2}a_{i_3}a_{i_4})))$, если он существует. На этом шаге потребуется неравенство $T>3x$.

Потом определим, существует ли путь $l(S,h(g^k(a_{i_1}a_{i_2}a_{i_3}a_{i_4}a_{i_5})))$ и, действуя подобным образом, доберёмся до
$l(S,h(g^k(a_{i_1}a_{i_2}\dots a_{i_k})))$. На последнем шаге нам хватит оценки $T>x(m-1)$.
\end{proof}

\begin{lemma}\label{hare2}
Если размер антиоснастки $S,k$ хотя бы $C_8C_6$, то по $(S,k)$ можно найти антиоснастку $(\Evol(S),k)$.
\end{lemma}
\begin{proof}
По номеру $T-$разруленной схемы $S$ в периоде или предпериоде можно узнать следующую разруленную схему. 
Для определения антиоснастки $(Evol(S),k)$ осталось найти все основные пути.
 
Согласно \ref{wts}, размер $(S,k+1)$ хотя бы $C_6$. Следовательно, в антиоснастке $(S,k+1)$ для каждого источника ровно один основной путь.
Пусть $s$ -- основной путь антиоснастки $(Evol(S),k)$, соответствующий источнику $p$. 
Согласно \ref{f}, в облегчённой схеме $S$ пути $s$ соответствует путь $s'$, у которого (в необлегчённой схеме) будет такое же слово, как и у $s$. Из минимальности $s_1=l(S,h(g^k(p_i)))$ следует, что $s_1\sqsubseteq s'$.

Так как $|s_1|>1$, то, согласно \ref{f}, в $S$ можно найти минимальный путь, содержащий $s_1$, который соответствует пути с таким же словом в $\Evol(S)$. 
Этим путём, очевидно, и будет являться $s$.
\end{proof}

Теперь перейдём к построению алгоритма. Возьмём нумерованную схему Рози $S_0$ и подберём такое $k_0$, чтобы размер антиоснастки $(S_0,k_0)$ был больше, чем
$\Gamma =2C_6C_7^2C_8^{C_9}$ (если такого $k$ не существует, то алгоритм выдаёт ответ ``нет''.) 
Далее рассмотрим следующую последовательность антиоснасток $(S_i,k_i)$: если размер антиоснастки $(S_{i-1},k_{i-1})$ меньше, чем $\Gamma$, то
$k_{i}=k_{i+1}+1$, $S_i=S_{i-1}$. Если же размер антиоснастки $(S_i,k_i)$ больше либо равен $\Gamma$, то $k_{i}=k_{i-1}$, $S_i=\Evol(S_{i-1})$. 
Размер первой оснастки обозначим $x_0$.

Из утверждений леммы леммы \ref{wts}, следует, что размер антиоснасток этой последовательности не опускается ниже $C_6C_7$ и не поднимается выше $X=\max\{C_7x_0,C_7C_8\Gamma\}$. (Теперь можно указать $T$, возьмём его равным $2X\max_{a\in A} |g(a)|$. ) 
Следовательно, различных антиоснасток в этой последовательности не может быть больше, 
чем некоторое алгоритмически определяемое число $R$. Попытаемся построить первые $R+1$ член последовательности антиоснасток. 
Если для некоторых $k_i$ и $S_i$ соответствующая антиоснастка не существует, то некоторое рабочее слово не является подсловом $W$ и 
алгоритм выдаёт ответ ``нет''.

Если первые $R+1$ антиоснастки удалось построить, то среди них есть две одинаковых. Пусть алгоритм выдаст ответ ``да''. Докажем,
что этот ответ правильный. Согласно леммам \ref{hare} и \ref{hare2}, существование каждого
следующего члена последовательности антиоснасток, а также сам следующий член определяются по предыдущему члену последовательности 
(то есть без информации о словах на рёбрах схемы и о числе $k$.) Значит, последовательность антиоснасток не оборвётся ни на каком члене. 
Для доказательства корректности работы алгоритма осталось показать, что в последовательности $(S_i,k_i)$ числа $k_i$ неограниченно возрастают.

Предположим противное. Тогда из \ref{size} следует, что в бесконечной последовательности схем Рози, 
полученной итерированием детерменированной эволюции из некоторой схемы, масштабы схем в ограничены в совокупности. 
Следовательно, различных слов на рёбрах всех схем последовательности тоже конечное число.
А так как облегчённых схем в последовательности конечно, то и число самих схем Рози в последовательности конечно.

Для каждой схемы Рози рассмотрим множество слов её симметричных путей. Согласно \ref{f}, такое множество, построенное для $\Evol(S)$, 
является собственным подмножеством множества слов симметричных путей схемы $S$. 
Отсюда следует, что никакая схема не может быть получена из себя же итерированием операции $\Evol$. 
Получаем противоречие с конечностью числа схем в последовательности.

\end{document}